\documentstyle[12pt]{article}

\begin{document}

\newtheorem{theorem}{Theorem}[section]
\newtheorem{corollary}[theorem]{Corollary}
\newtheorem{definition}[theorem]{Definition}
\newtheorem{conjecture}[theorem]{Conjecture}
\newtheorem{question}[theorem]{Question}
\newtheorem{lemma}[theorem]{Lemma}
\newtheorem{proposition}[theorem]{Proposition}
\newtheorem{example}[theorem]{Example}
\newenvironment{proof}{\noindent {\bf
Proof.}}{\rule{3mm}{3mm}\par\medskip}
\newcommand{\remark}{\medskip\par\noindent {\bf Remark.~~}}
\newcommand{\pp}{{\it p.}}
\newcommand{\de}{\em}

\newcommand{\JEC}{{\it Europ. J. Combinatorics},  }
\newcommand{\JCTB}{{\it J. Combin. Theory Ser. B.}, }
\newcommand{\JCT}{{\it J. Combin. Theory}, }
\newcommand{\JGT}{{\it J. Graph Theory}, }
\newcommand{\ComHung}{{\it Combinatorica}, }
\newcommand{\DM}{{\it Discrete Math.}, }
\newcommand{\ARS}{{\it Ars Combin.}, }
\newcommand{\SIAMDM}{{\it SIAM J. Discrete Math.}, }
\newcommand{\SIAMADM}{{\it SIAM J. Algebraic Discrete Methods}, }
\newcommand{\SIAMC}{{\it SIAM J. Comput.}, }
\newcommand{\ConAMS}{{\it Contemp. Math. AMS}, }
\newcommand{\TransAMS}{{\it Trans. Amer. Math. Soc.}, }
\newcommand{\AnDM}{{\it Ann. Discrete Math.}, }
\newcommand{\NBS}{{\it J. Res. Nat. Bur. Standards} {\rm B}, }
\newcommand{\ConNum}{{\it Congr. Numer.}, }
\newcommand{\CJM}{{\it Canad. J. Math.}, }
\newcommand{\JLMS}{{\it J. London Math. Soc.}, }
\newcommand{\PLMS}{{\it Proc. London Math. Soc.}, }
\newcommand{\PAMS}{{\it Proc. Amer. Math. Soc.}, }
\newcommand{\JCMCC}{{\it J. Combin. Math. Combin. Comput.}, }
\newcommand{\GC}{{\it Graphs Combin.}, }

\title{The Laplacian Eigenvalues and Invariants of Graphs\thanks{
Supported by National Natural Science Foundation of China
(No:11271256).}}
\author{  Rong-Ying Pan\\
{\small School of Science, Suzhou Vocational University}\\
{\small Suzhou, Jiangsu 215004, P.R.China}\\
Jing Yan\\
{\small School of Science, Jiangsu University of Technology}\\
{\small Changzhou, Jiangsu, 213001, P.R.China}\\
Xiao-Dong Zhang \\
{\small Department of Mathematics, and MOE-LSC, Shanghai Jiao Tong University} \\
{\small  1954 Huashan road, Shanghai,200030, P.R. China}\\
{\small Email:  xiaodong@sjtu.edu.cn}
 }
\date{}
\maketitle
 \begin{abstract}
 In this paper, we  investigate some relations between the invariants (including vertex
  and edge connectivity and forwarding indices)  of a graph and its Laplacian eigenvalues.
 In addition, we present a sufficient condition for the existence
 of Hamiltonicity  in a graph involving its  Laplacian
 eigenvalues.
 \end{abstract}

\vskip 0.5cm

\section{introduction}

Let $G = (V,~E)$ be a simple graph with  vertex set
$V(G)=\{v_1,\cdots, v_n\}$ and edge set $E(G)=\{e_1,\cdots,
e_m\}$. Denote by $d(v_i)$ the {\it degree} of vertex $v_i$. If
$D(G)=diag(d_u, u\in V)$ is the diagonal matrix of vertex degrees
of $G$ and $A(G)$ is the $0-1$ {\it  adjacency matrix} of $G$, the
matrix $L(G)=D(G)-A(G)$ is called the {\it Laplacian matrix} of a
graph $G$ Moreover, the eigenvalues of $L(G)$ are called  {\it
Laplacian eigenvalues} of $G$. Furthermore, the Laplacian
eigenvalues of $G$ are denoted by

 $$0=\sigma_0\le \sigma_1\le\cdots\le \sigma_{n-1},$$
 since $L(G)$ is  positive semi-definite.
In recent years, the relations between invariants of a graph and
its Laplacian eigenvalues have been investigated extensively. For
example, Alon in \cite{alon1986} established that there are
relations between an expander of a graph and its second smallest
eigenvalue; Mohar in \cite{mohar1992} presented a necessary
condition foe the existence of Hamiltonicity in a graph in terms
of its Laplacian eigenvalues. The reader is refereed to
\cite{chung1997}, \cite{heydemann1989} and \cite{Merris1994} etc.

  The purpose of this paper is to present some relations between
  some invariants of a graph and its Laplacian eigenvalues. In
  Section 2, the relations between the vertex and edge
  connectivities  of a graph and its Laplacian eigenvalues are
  investigated. In Section 3, we present  a sufficient condition
  for the existence of Hamiantonicity in a graph involving its
Laplacian  eigenvalues. In last Section, the lower bounds for
forwarding indices of networks are obtained.  Before finishing
this section, we present a general discrepancy inequality from
Chung\cite{chung2004}, which is very useful for later.

 For a subset $X$ of vertices in $G$, the {\it volume} ${\rm vol}(X)$ is
 defined by ${\rm vol}(X)=\sum_{v\in X}d_v$, where $d_v$ is the
 degree of $v$.  For any two subsets $X$ and $Y$ of vertices in
 $G$,  denote
 $$e(X,Y)=\{(x,y): x\in X, y\in Y, \{x, y\}\in E(G)\}.$$
 \begin{theorem}
 \label{disc}
 Let $G$ be a simple graph with $n$ vertices and average degree
 $d=\frac{1}{n}{\rm vol}(G)$. If the Laplacian eigenvalues
 $\sigma_i$ of $G$ satisfy $|d-\sigma_i|\le \theta$ for
 $i=1,2,\cdots, n-1$, then for any two subsets $X$ and $Y$ of
 vertices in $G$,  we have
 $$|e(X,Y)-\frac{d}{n}|X||Y|+d|X\bigcap Y|-{\rm vol}(X\bigcap
 Y)|\le\frac{\theta}{n}\sqrt{|X|(n-|X|)|Y|(n-|Y|}.$$
 \end{theorem}

\section{Connectivity}
The vertex connectivity of a graph $G$ is the minimum number of
vertices that we need to delete to make $G$ is disconnected and
denoted by $\kappa(G)$. Fiedler in \cite{fiedler} proved that if
$G$ is not the complete graph, then $\kappa(G)$ is at least the value of the
second smallest Laplacian eigenvalue. In here, we present another
bound for the vertex connectivity of a graph.

\begin{theorem}
 \label{vertexcon}
Let $G$ be a simple graph  of order $n$ with the smallest degree
$\delta\le \frac{n}{2}$ and average degree $d$. If the Laplacian
eigenvalues $\sigma_i$ satisfies $|d-\sigma_i|\le \theta$ for
$i\neq 0$, then
$$\kappa(G)\ge \delta-(2+2\sqrt{3})^2\frac{\theta^2}{\delta}.$$
 \end{theorem}
 \begin{proof}
 Let $c=2+2\sqrt{3}$. If $\theta\ge \frac{\delta}{c}$, there is
 nothing to show. We assume that $\theta<\frac{\delta}{c}$.

  Suppose that there exists a subset $S\subset V(G)$ with
  $|S|<\delta-\frac{(c\theta)^2}{\delta}$ such that the induced
  graph $G[V\setminus S]$ is disconnected. Denote by $U$ the set of
  vertices of the smallest connected component of $G[V\setminus
  S]$ and $W=V\setminus (S\bigcup U)$. Since the smallest degree
  of $G$ is $\delta$, $|S|+|U|>\delta$, which implies $|U|\ge
  \frac{(c\theta)^2}{\delta}$. Moreover,
  $|W|= n-(|U|+|S|)\le \frac{n-\delta}{2}\le \frac{n}{4}$.
  Because $U$ and $W$ are disjoint for two subsets of $G$, by
  \ref{disc}, we have
  $$\frac{d}{n}|U||W|\le
  \frac{\theta}{n}\sqrt{|U||W|(n-|U|)(n-|W|)}\le \sqrt{|U||W|}.$$
  Hence
  $$|U|\le \frac{\theta^2n^2}{d^2|W|}\le
  \frac{\theta}{d}\frac{n}{|W|}\frac{\theta
  n}{d}<\frac{4}{c}\frac{\theta n}{d}, $$
  since $\frac{\theta}{d}<\frac{\theta}{\delta}<\frac{1}{c}$.
  By using Corollary 4 in \cite{chung2004}, we have
   $$|2|e(U)|-\frac{d|U|(|U|-1)}{n}|\le
   \frac{2\theta}{n}|U|(n-\frac{|U|}{2}). $$
Then
\begin{eqnarray*}
 2|e(U)|&\le & 2\theta|U|+\frac{d}{n}|U|^2\\
 &\le & (2\theta+\frac{d}{n}\frac{4}{c}\frac{\theta n}{d})|U|\\
 &=&(2+\frac{4}{c})\theta |U|.
 \end{eqnarray*}
Hence, by $\theta<\frac{\delta}{c}$ and $c=2+2\sqrt{3},$
\begin{eqnarray*}
|e(U,S)|&\ge & \delta |U|-2|e(U)|\\
&\ge & (\delta-(2+\frac{4}{c})\theta)|U|\\
>(1-(2+\frac{4}{c}\frac{1}{c}))\delta |U|\\
>(\frac{1}{2}+\frac{1}{c})\delta |U|.
\end{eqnarray*}
On the other hand,  by \ref{disc} and $|S|\le \delta, $  $|U|\ge
\frac{c^2\theta^2}{\delta}$ and $\frac{d}{n}\le \frac{1}{2}$, we
have
\begin{eqnarray*}
|e(U,S)|&\le & \frac{d}{n}|U||S|+\theta\sqrt{|U||S|}\\
&\le & (\frac{d\delta}{n}+\theta
\frac{\delta\sqrt{\delta}}{c\theta})|U|\\
&=&(\frac{1}{2}+\frac{1}{c})\delta |U|.
\end{eqnarray*}
It is a contradiction. Therefore the result holds.
 \end{proof}
 \begin{corollary}
 \label{Krivvertex}
 (\cite{krivelevich2004}) Let $G$ be a $d-$regular graph of order
 $n$ with $d\le \frac{n}{2}$. Denote by  $\lambda$ the second largest
 absolute eigenvalue of $A(G)$. Then
 $$\kappa(G)\ge d-\frac{36\lambda^2}{d}.$$
 \end{corollary}
 \begin{proof}
 Since $G$ is a $d-$regular graph, the eigenvalues of $A(G)$ are
 $d-\sigma_0$, $d-\sigma_1$,$\cdots,$  $d-\sigma_{n-1}$. Hence
 $\lambda$ satisfies $|d-\sigma_i|\le \lambda$ for $i\neq 0$. It
 follows from Theorem~\ref{vertexcon} that $\kappa(G)\ge
 d-\frac{(2+2\sqrt{3})^2d^2}{d}\ge d -\frac{36\lambda^2}{d}$.
 \end{proof}

  From \cite{krivelevich2004}, for a $d-$regular graph, the lower bound for
  $\kappa(G)$ in Corollary~\ref{Krivvertex} is tight up to a
  constant factor, which implies Theorem~\ref{vertexcon} is tight
  up to a constant factor.

 It is known that the edge connectivity $\kappa^{\prime}(G)$  of a
  graph $G$ is the minimum number of edges that need to delete  to
  make disconnected. In \cite{goldsmith1979}, Goldsmith and
  Entringer gave a sufficient condition for edge connectivity
  equal to the smallest degree. In here, we present also  a
  sufficient condition for edge connectivity equal to the smallest
  degree in terms of its Laplacian eigenvalues.

  \begin{theorem}
  \label{edgecon}
  Let $G$ be a graph of order $n$ with average degree $d$ and the
  smallest degree $\delta$. If the Laplacian eigenvalues satisfy
  $2\le \sigma_1\le \sigma_{n-1}\le 2d-2$, then
  $\kappa^{\prime}(G)=\delta$.
  \end{theorem}
  \begin{proof}
  Let $U$ be a subset of vertices of $G$ with $|U|\le
  \frac{n}{2}$.

  If $1\le |U|\le \delta$, then for every vertex $u\in U$, $u$ is
adjacent to at least $\delta-|U|+1$ vertices in $G\setminus U$.
Therefore,
$$|e(U,G\setminus U)|\ge |U|(\delta-|U|+1)\ge \delta.$$

  If $\delta<|U|\le \frac{n}{2}$, let $\theta=d-2$. Since
   $2\le \sigma_1\le \sigma_{n-1}\le 2d-2$,
  $|d-\sigma_i|\le \theta$ for $i\neq 0$. By Theorem~\ref{disc},
  $$||e(U, V\setminus U)|-\frac{d}{n}|U||V\setminus U||\le
  \frac{\theta}{n}|U|(n-|U|).$$
  Thus,
  $$|e(U, V\setminus U)|\ge \frac{d-\theta}{n}|U|(n-|U|)\ge
\frac{d-\theta}{n}\delta(n-\delta)\ge
\frac{2\delta(n-\delta)}{n}\ge \delta.$$ Hence there are always at
least $\delta$ edges between $U$ and $V\setminus U$. Therefore
$\kappa^{\prime}(G)=\delta$.
\end{proof}

\section{Hamiltonicity and the chromatic number}
 In this section, we first give an upper bound for the
 independence number $\alpha(G)$, which is used to present a sufficient
 condition for a graph to have a Hamilton cycle. Moreover,  a
 lower bound for the chromatic number of a graph is obtained.
The independence number is the maximum cardinality of a set of
vertices of $G$ no two of which are adjacent.

 \begin{lemma}
 \label{independ}
  Let $G$ be a graph of order $n$ with average
 $d$. If the Lapalcian eigenvalues satisfies $|d-\sigma_i|\le
 \theta$ for $i\neq 0$, then
 $$\alpha(G)\le \frac{2n\theta+d}{d+\theta}.$$
 \end{lemma}
\begin{proof}
Let $U$ be an independent set with the seize $\alpha(G)$. By
Corollary~4 in \cite{chung2004}, we have
$$|2|e(U)|-\frac{d|U|(|U|-1)}{n}|\le
\frac{2\theta}{n}|U|(n-\frac{|U|}{2}).$$ Hence
$|U|\le\frac{2n\theta+d}{d+\theta}$.
\end{proof}
\begin{lemma}
\label{chvatal} \cite{chvatal} Let $G$ be a graph. If the vertex
connectivity of $G$ is at least as large as its independence
number, then $G$ is Hamiltonian.
\end{lemma}

 \begin{theorem}
\label{Hamilton1} Let $G$ be a  graph  of order $n$ with average
$d$ and the smallest degree $\delta$. If the Laplacian eigenvalues
satisfies
 $|d-\sigma_i|\le
 \theta$ for $i\neq 0$ and $\delta-(2+2\sqrt{3})^2\frac{\theta^2}{\delta}\ge \frac{2n\theta+d}{d+\theta}
$, then $G$ is Hamiltonian.
\end{theorem}
\begin{proof}
By Theorem~\ref{vertexcon}, $G$ has at least
$\delta-(2+2\sqrt{3})^2\frac{\theta^2}{\delta}$ vertex connected.
On the other hand, by
 Lemma~\ref{independ}, the independence number of $G$ is at most
 $\frac{2n\theta+d}{d+\theta}$. It follows from Lemma~\ref{chvatal} that $G$ is Hamiltonian.
 \end{proof}

\begin{theorem}
\label{hamilton2} Let $G$ be a  connected graph of order $n$ with
the smallest degree $\delta$. If $\sigma_1\ge
\frac{\sigma_{n-1}-\delta}{\sigma_{n-1}}n$, then $G$ is
Hamiltonian.
\end{theorem}
\begin{proof}
By a theorem in \cite{fiedler}, $\kappa(G)\ge \sigma_1$. On the
other hand, by Corollary~3.3 in \cite{zhang2004}, the independence
number $\alpha(G)\le \frac{\sigma_{n-1}-\delta}{\sigma_{n-1}}n$.
It follows from Lemma~\ref{chvatal} that $G$ is Hamiltonian.
\end{proof}

 The proper coloring of the vertices of $G$  is an assignment  of
 colors to the vertices in such a way that adjacent  vertices have
 distinct colors.  The chromatic number, denoted by $\chi(G)$, is the
 minimal number od colors in a vertex coloring of $G$.

\begin{theorem}
\label{Chromatic} Let $G$ be a graph of order $n$ with the
smallest degree $\delta\ge 1$. Then
$$\chi(G)\ge
\frac{\sigma_{n-1}}{\sigma_{n-1}-\delta}.$$ Moreover, if $G$ is a
$d-$
 regular bipartite graph, or a complete $r-$partite graph
 $K_{s,s,\cdots, s}$, then equality holds.
 \end{theorem}
 \begin{proof}
 Let $V_1,V_2,\cdots, V_{\chi}$ denote the color class of $G$.
 Denote  by $e$ the vector with all component equal to $1$. Let
 $s_i$ be the restriction  vector of $\frac{1}{|V_i|}e$ to $V_i$;
 that is,
 $(s_i)_j=\frac{1}{|V_i|}$, if $j\in V_i$ ; $(s_i)_j=0$,
 otherwise. Thus $S=(s_1,\cdots, s_{\chi})$ is an $n\times \chi$
 matrix and $S^TS=I_n$. Let $B=S^TL(G)S=(b_{ij})$ and its
 eigenvalues $\mu_0\le \mu_1\le\cdots\le \mu_{\chi -1}$.
 By eigenvalue interlacing, it is easy to see that $\mu_0=0$ and
 $\mu_{\chi-1}\le \sigma_{n-1}$. Moreover,
 $b_{ii}=\frac{1}{|V_i|}\sum_{v\in V_i}d_v\ge \delta$. Hence
 $$\delta\chi\le {\rm tr}B= \mu_0+\cdots \mu_{\chi-1}\le
 (\chi-1)\sigma_{n-1},$$
which yields the desired inequality.  If $G$ is a $d-$ regular
graph, then $\chi=2$, $\delta=d$ and $\sigma_{n-1}=2d$. So
equality holds. If $G$ is a complete $r-$partite graph, then
$\chi=r$, $\delta=(r-1)s$ and $\sigma_{n-1}=\frac{r}{r-1}s$. Hence
equality holds.\end{proof}

\section{Forwarding indices of  graphs}
In this section, we discuss some relations between the Laplacian
eigenvalues of a graph and its forwarding indices.

 A routing $R$ of a graph $G$ of order $n$ is a set of $n(n-1)$
 paths specified for all ordered pairs $u$ and $v$ of vertices of
 $G$. Denote $\xi(G,R, v)$ by the number of paths of $R$ going
 through $v$ (where $v$ is not an end vertex). The vertex
 forwarding index of $G$ is defined to be
 $$\xi(G)=\min_{R}\max_{v\in V(V)} \xi(G,R,v).$$
 Denote  $\pi(G,R, e)$ by the number of paths of $R$ going through
 edge $e$. The edge forwarding index of $G$ is defined to be
 $$\pi(G)=\min_{R}\max_{e\in E(G)}\pi(G,R, e).$$
 Let $X$ be a proper subset of $V$. The vertex cut induced by $X$
 is $N(X)=\{y\in V\setminus X | \{x,y\}\in E(G)\}$. Moreover,
 denote $X^+$ by the complement of $X\bigcup N(X)$ in $V$. The
 {\it vertex expanding factor} is defined by
 $$\gamma(G)=\min\{\frac{|N(X)|}{|X||X^+|}\ | \ X\subseteq V, 1\le
 |X|\le n-1, |X^+|\ge 1\}, $$
 where the min on a void set of $X$ is taken to be infinite.

 \begin{theorem}
 \label{verforward}
 Let $G$ be a graph of order $n$ with average degree $d$. If the
 Laplacian eigenvalues satisfies $|d-\sigma_i|\le \theta$ for
 $i\neq 0$, then
 $$\gamma(G)\ge \frac{d^2-\theta^2}{n\theta^2}.$$
 \end{theorem}
 \begin{proof}
 Let $U$ be a subset of $G$ such that
 $$\gamma(G)=\frac{|N(U)|}{|U||U^+|},\ \ 1\le |U|\le n-1,\ \
 |U^+|\ge 1.$$
 Set $W=V\setminus (U\bigcup N(U))$. By Theorem~\ref{disc}, we
 have
 $$||e(U,W)|-\frac{d}{n}|U||W||\le
 \frac{\theta}{n}\sqrt{|U|(n-|U|)|W|(n-|W|)}.$$
 Hence
 $$ d^2|U||W|\le \theta^2(|U|+|N(U)|)(|W|+|N(W)|).$$
 Then
 $$\frac{|N(U)|}{|U||U^+|}=\frac{|N(U)|}{|U|(n-|W|)}\ge
 \frac{d^2-\theta^2}{n\theta^2}.$$
 We complete the proof.
\end{proof}

\begin{theorem}
\label{edgeindex}
 Let $G$ be a graph of order $n$. If $\sigma_1\le \frac{1}{2}$,
 then $\xi(G)\ge \sqrt{\frac{1-2\sigma_1}{\sigma_1}}$.
 \end{theorem}
\begin{proof}
By Lemma~2.4 in \cite{alon1986}, we have
$$\sigma_1\ge \frac{c^2}{4+2c^2},$$
where $c$ satisfies $ \frac{|N(X)|}{|X|}\ge c$ for every $|X|\le
\frac{n}{2}$ and $X\subset U$. Hence
$$\gamma(G)\le c\le \sqrt{\frac{4\sigma_1}{1-2\sigma_1}}.$$
On the other hand, there exists a subset $U$ such that
$\gamma(G)=\frac{|N(U)|}{|U||U^+|}$. It follows from the
definition of $\xi(G)$ that $2|U||U^+|\ge \xi(G) |N(U)|$, since
there does not exist edges between $U$ and $U^+$. Hence
$$\xi(G)\ge \frac{2|U||U^+|}{|N(U)|}=\frac{2}{\gamma(G)}
\ge \sqrt{\frac{1-2\sigma_1}{\sigma_1}}.$$ We finish the proof.
\end{proof}
\begin{lemma}
\label{edgecut} Let $G$ be a graph of order $n$ with average
degree $d$ and let   $\beta(G)=\min \{\frac{|e(U, V\setminus
U)|}{|U|(n-|U|)}, \ \ 1\le |U|\le n-1\}$.   If the Laplacian
eigenvalues satisfy $|d-\sigma_i||le \theta$ for $i\neq 0$, then
$$\beta(G)\le \frac{d+\theta}{n}.$$
\end{lemma}
\begin{proof}
By the definition of $\beta(G)$, there exists a subset $U$ such
that $\beta(G)=\frac{|e(U, V\setminus U)|}{|U|(n-|U|)}$. On the
other hand, by Theorem~\ref{disc},  we have
$$||e(U,V\setminus U)|-\frac{d}{n}|U|(n-|U|)|le
\frac{\theta}{n}|U|(n-|U|).$$
 Hence $\beta(G)\le \frac{d+\theta}{n}$.
 \end{proof}

\begin{theorem}
\label{edgeindices} Let $G$ be a graph of order $n$ with average
degree $d$. If the Laplacian eigenvalues satisfy $|d-\sigma_i|\le
\theta$ for $i\neq 0$, then
$$\pi(G)\ge\frac{2n}{d+\theta}.$$
\end{theorem}
\begin{proof}
It follows from  Theorem~1 $\pi(G)\beta(G)\ge 2$ in
\cite{sole1995} and Lemma~\ref{edgecut} that the result holds.
\end{proof}

{\bf Remark} The lower bounds for $\xi(G) $ and $\pi(G)$ are tight
up to a constant factor. For example, Let $P_n$ be a path of order
$n$. It is easy to see that
$\xi(P_n)=2(\lfloor\frac{n}{2}\rfloor(\lceil\frac{n}{2}\rceil-1)$,
$\pi(G)=2\lfloor\frac{n}{2}\rfloor\lceil\frac{n}{2}\rceil$; while
$\sigma_1=4\sin^2\frac{\pi}{2n}$.

\end{document}